\documentclass[a4paper,12pt]{article}
\usepackage[utf8]{inputenc}
\usepackage{amsfonts}
\usepackage{amssymb}
\usepackage{latexsym}
\usepackage{amsmath}
\usepackage{amsthm}
\usepackage{graphicx}

\def\ZZ{\mathbb{Z}}
\def\RR{\mathbb{R}}
\def\QQ{\mathbb{Q}}
\def\NN{\mathbb{N}}

\newcommand{\bsalpha}{\boldsymbol{\alpha}}
\newcommand{\bsh}{\boldsymbol{h}}
\newcommand{\bsg}{\boldsymbol{g}}
\newcommand{\bsk}{\boldsymbol{k}}

\newcommand{\bsm}{\boldsymbol{m}}
\newcommand{\bsx}{\boldsymbol{x}}
\newcommand{\bsy}{\boldsymbol{y}}
\newcommand{\bsz}{\boldsymbol{z}}
\newcommand{\bszero}{\boldsymbol{0}}
\newcommand{\bsell}{\boldsymbol{\ell}}

\newcommand{\cP}{\mathcal{P}}
\newcommand{\rd}{\,{\rm d}}

\usepackage{enumerate,color}

\newtheorem{theorem}{Theorem}
\newtheorem{lemma}[theorem]{Lemma}

\newtheorem{remark}[theorem]{Remark}
\newtheorem{definition}[theorem]{Definition}

\newcommand{\abs}[1]{\Big\vert #1 \Big\vert}	
\newcommand{\mr}[1]{\mathring{#1}}	
\newcommand{\ol}[1]{\overline{#1}}
\newcommand{\ls}{\langle}
\newcommand{\rs}{\rangle}
\newcommand{\lsb}{\Big\langle}
\newcommand{\rsb}{\Big\rangle}

\newcommand{\bsxi}{\boldsymbol{\xi}}

\newcommand{\R}{\mathbb{R}}
\newcommand{\N}{\mathbb{N}}
\newcommand{\Z}{\mathbb{Z}}
\newcommand{\C}{\mathbb{C}}
\newcommand{\X}{\mathbb{X}}
\renewcommand{\S}{\mathcal{S}}
\newcommand{\eps}{\varepsilon}
\renewcommand{\a}{\alpha}
\renewcommand{\phi}{\varphi}
\newcommand{\F}{\mathcal{F}} 
\newcommand{\supp}{{\rm supp}}
\newcommand{\sinc}{{\rm sinc}}
\newcommand{\B}{\ensuremath{\mathbf{B}}}
\newcommand{\Bspt}{\ensuremath{\mathbf{B}^s_{p,\theta}}}
\newcommand{\Bo}{\ensuremath{\mr{\mathbf{B}}^s_{p,\theta}}}
\renewcommand{\H}{\ensuremath{\mathbf{H}}}
\newcommand{\Hsp}{\ensuremath{\mathbf{H}^s_{p}}}
\newcommand{\Hs}{\ensuremath{\mathbf{H}^s_{2}}}
\newcommand{\Ho}{\ensuremath{\mr{\mathbf{H}}^s_{p}}}

\allowdisplaybreaks

\title{Lattice based integration algorithms: Kronecker sequences and rank-1 lattices}

\author{Josef Dick, Friedrich Pillichshammer, Kosuke Suzuki, Mario Ullrich, \\ and Takehito Yoshiki\thanks{The research of J.\ Dick, K. Suzuki and T. Yoshiki was supported under the Australian Research Councils Discovery Projects funding scheme (project number DP150101770). F. Pillichshammer was supported by the Austrian Science Fund (FWF): Projects F5509-N26 (Pillichshammer), which is part of the Special Research Program ``Quasi-Monte Carlo Methods: Theory and Applications''.} 
}

\begin{document}

\maketitle

\begin{abstract}
We prove upper bounds on the order of convergence of lattice based algorithms for numerical integration in function spaces of dominating mixed smoothness on the unit cube with homogeneous boundary condition. More precisely, we study worst-case integration errors for Besov spaces of dominating mixed smoothness $\Bo$, which also comprise the concept of Sobolev spaces of dominating mixed smoothness $\Ho$ as special cases. The considered algorithms are quasi-Monte Carlo rules with underlying nodes from $T_N(\Z^d) \cap [0,1)^d$, where $T_N$ is a real invertible generator matrix of size $d$. For such rules the worst-case error can be bounded in terms of the Zaremba index of the lattice $\X_N=T_N(\Z^d)$. We apply this result to Kronecker lattices and to rank-1 lattice point sets, which both lead to optimal error bounds up to $\log N$-factors for arbitrary smoothness $s$.
The advantage of Kronecker lattices and classical lattice point sets is that the run-time of algorithms generating these point sets is very short.
\end{abstract}

\centerline{\begin{minipage}[hc]{150mm}{
}
\end{minipage}}

\section{Introduction}

In this paper we study numerical integration of smooth functions on the $d$-dimensional unit cube which satisfy homogeneous boundary conditions. More precisely, we consider Besov spaces of dominating mixed smoothness $\Bo$ which also comprise the concept of Sobolev spaces of dominating mixed smoothness $\Ho$ as special cases. The exact definition of these spaces requires some preparation and will be given in Section~\ref{sec:spaces}. For the moment we just note that the parameter $s$ denotes the underlying smoothness of the functions and that for $f \in \Bo$ we have ${\supp}(f) \subset [0,1]^d$.

We study numerical integration $$I_d(f)=\int_{[0,1]^d}f(\bsx)\rd \bsx  \ \ \ \mbox{ for $f \in\Bo$}$$ with linear algorithms of the form $$Q_{N,d}(f)=\sum_{j=0}^{N-1} a_j f(\bsx_j)$$ for given sample points $\bsx_0,\ldots,\bsx_{N-1} \in [0,1]^d$ and real weights $a_0,\ldots,a_{N-1}$. For the special choice $a_j=1/N$ for all $j=0,\ldots,N-1$ we speak of quasi-Monte Carlo (QMC) algorithms.

The {\it worst-case error} of an algorithm $Q_{N,d}$ is the worst absolute 
integration error of $Q_{N,d}$ over the unit-ball of $\Bo$, i.e. 
\[
{\rm wce}(Q_{N,d},\Bo)=\sup_{\|f\|_{\Bo} \le 1} \left|I_d(f) -Q_{N,d}(f)\right|,
\] 
where $\|\cdot\|_{\Bo}$ denotes the norm in $\Bo$.

As integration rules we use what we call {\it general lattice rules} which are of the form 
\begin{equation}
Q_{T}(f) \,=\, |\det(T)| \sum_{\bsx\in T(\Z^d)\cap[0,1)^d} f(\bsx)
\end{equation}
with an invertible matrix $T\in\R^{d\times d}$. Note that $$\mathbb{X}=T(\mathbb{Z}^d)=\{T\bsx \ : \ \bsx \in \mathbb{Z}^d\}$$ is a $d$-dimensional lattice in $\R^d$, i.e., a discrete additive subgroup of $\mathbb{R}^d$ not contained in a proper linear subspace of $\mathbb{R}^d$.

A specific example of such a rule is Frolov's cubature formula introduced in \cite{Frolov} (see also \cite{U16} and the survey paper by Temlyakov~\cite[Section~4.3]{teml2003}). Here a suitable generating matrix $T$ for the underlying lattice $\mathbb{X}$ is found as follows: Define the polynomial $p_d \in \mathbb{R}[x]$, $p_d(x)=-1+\prod_{j=1}^d (x-2j+1)$. Then $p_d$ has only integer coefficients, it is irreducible over the rationals and has $d$ different real roots, say $\xi_1,\ldots,\xi_d \in \mathbb{R}$. With these roots define the $d \times d$ Vandermonde matrix $B$ by $$B=(B_{i,j})_{i,j=1}^d =(\xi_i^{j-1})_{i,j=1}^d.$$ Then the generator of the lattice $\mathbb{X}$ used in Frolov's cubature formula is $$T=\frac{1}{a}(B^{\top})^{-1},$$ where $a >1$ is a shrinking factor. 

It was shown by Dubinin~\cite{Du2}, see also \cite{Te93} and \cite{UU16}, 
that Frolov's cubature formula  achieves the optimal convergence rate for 
the worst-case error in $\Bo$, which is 
\[
{\rm wce}(Q_T,\Bo) \asymp \frac{(\log N)^{(d-1)(1-1/\theta)}}{N^s},
\] 
where $N$ is the number of elements of $T(\Z^d)$ that belong to the unit 
cube $[0,1]^d$.
The lower bound for these spaces was proven in \cite{U16}. 
See also \cite{NUU16} for techniques to transfer such results to spaces 
without (or periodic) boundary conditions.
For the Sobolev(-Hilbert) spaces $\mr{\H}^s_2=\mr{\B}^s_{2,2}$ 
the result reads
\[
{\rm wce}(Q_T,\mr{\H}^s_2) \asymp \frac{(\log N)^{(d-1)/2}}{N^s}.
\] 

The problem with Frolov's cubature formula is that one needs to determine
 which points from the shrunk lattice belong to the unit cube $[0,1]^d$. 
This is in general a very difficult task, especially when the dimension $d$ 
is large\footnote{It is known that $N \sim |\det(T^{-1})|$ as the shrinking 
factor $a$ tends to infinity. This can be quantified as 
$\bigl|N- |\det(T^{-1})|\bigr| \lesssim \log^{d-1}(1+a^d)$ for all $a>1$, 
see Skriganov~\cite[Theorem~1.1]{skri}.}.\\

It is the aim of this paper to find other general lattice rules whose 
lattice points are faster to generate on a computer also for large 
dimensions $d$ and which also achieve optimal convergence rate for 
the worst-case error with respect to the smoothness $s$, up to $\log N$-terms.

We will prove a very general estimate for the worst-case error in terms of the Zaremba index of the lattice $\X=T(\Z^d)$, see Theorem~\ref{thm:general2} in Section~\ref{sec:genbd}. Then we apply this error estimate to two examples of general lattice rules, to so-called Kronecker lattices (Subsection~\ref{subsec:Kronecker}) and to rank-1 lattice point sets (Subsection~\ref{subsec:rank-1}). In both cases we achieve a convergence rate of order $O(N^{-s})$ for the worst-case error up to $\log N$-terms. The major advantage of these constructions is that it is automatically clear which points belong to the unit cube. Hence the proposed integration algorithms (which are in fact QMC rules)  can be implemented very efficiently.

\paragraph{Basic notation:} 
Throughout the paper $d \in \N$ denotes the dimension. 
By $\log$ we denote 
the dyadic logarithm. For $x \in \R$ we denote by $\{x\}$ the fractional part of $x$ and by $\ls x\rs$ the distance of $x$ to the nearest integer, i.e., $\ls x \rs:=\min_{m\in\Z}|x-m|$.

For $\bsell=(\ell_1,\ldots,\ell_d) \in \mathbb{N}_0^d$ we denote by
$$D^{\bsell} f=\frac{\partial^{\,|\bsell|}}{\partial x_1^{\ell_1}\ \partial x_2^{\ell_2}\ \cdots\ \partial x_d^{\ell_d}}f$$ the operator of partial differentiation with $|\bsell|=\ell_1+\cdots+\ell_d$. Furthermore, for $\bsxi \in \R^d$ let $e_{\bsxi}:\R^d \rightarrow \C$, $e_{\bsxi}(\bsx):=\exp(2 \pi \texttt{i}  \langle \bsxi , \bsx\rangle)$, where $\langle \cdot, \cdot\rangle$ denotes the usual inner product in $\R^d$.

Before we formulate the main results of this work in more detail we need some preparation, which is presented in Section~\ref{Prel}.

\section{Preliminaries}
\label{Prel}

\subsection{Basics of Fourier analysis}
\noindent
Let $L_p=L_p(\R^d)$, $0 < p\le\infty$, be the space of all measurable functions 
$f:\R^d\to\C$ such that 
\[
\|f\|_p := \Big(\int_{\R^d} |f(\bsx)|^p \rd \bsx \Big)^{1/p} < \infty,
\]
with the usual modification if $p=\infty$. 
We will also need $L_p$-spaces on compact domains $\Omega\subset\R^d$ 
instead of the entire $\R^d$. 
We write $\|f\|_{L_p(\Omega)}$ for the corresponding (restricted) $L_p$-norm. 

For $f\in L_1(\R^d)$ we define the Fourier transform
\[
\F f(\bsxi) 
\,=\, \int_{\R^d} f(\bsy)\, \overline{e_{\bsxi}(\bsy)} \rd \bsy, \qquad \bsxi\in\R^d,  
\]
and the corresponding inverse Fourier transform $\F^{-1}f(\bsxi)=\F f(-\bsxi)$.
Additionally, we define the spaces of 
continuous functions $C(\R^d)$, infinitely differentiable functions $C^\infty(\R^d)$ 
and infinitely differentiable functions with compact support $C^\infty_0(\R^d)$ 
as well as 
the \emph{Schwartz space} $\S=\S(\R^d)$ of 
all rapidly decaying infinitely differentiable functions on $\R^d$, i.e.,  
\[
\S := \bigl\{\varphi\in C^{\infty}(\R^d)\colon \|\varphi\|_{\bsk,\bsell}<\infty 
\;\text{ for all } \bsk,\bsell\in \N^d_0 \bigr\}\,,
\]
where
$$
\|\varphi\|_{\bsk,\bsell}:=\Big\| \prod_{j=1}^d (1+|x_j|)^{k_j}  |D^{\bsell} \varphi(\bsx) |\Big\|_{\infty}\,.
$$

The space $\mathcal{S}'(\R^d)$, the topological dual of $\mathcal{S}(\R^d)$, is also referred to as the set of tempered
distributions on $\R^d$. Indeed, a linear mapping $f:\mathcal{S}(\R^d) \to \C$ belongs
to $\mathcal{S}'(\R^d)$ if and only if there exist vectors $\bsk,\bsell \in \N_0^d$
and a constant $c = c_f$ such that
\begin{equation}\label{eq100}
    |f(\varphi)| \leq c_f\|\varphi\|_{\bsk,\bsell}
\end{equation}
for all $\varphi\in \mathcal{S}(\R^d)$. The space $\mathcal{S}'(\R^d)$ is
equipped with the weak$^{\ast}$-topology.
The convolution $\varphi\ast \psi$ of two
square-integrable
functions $\varphi, \psi$ is defined via the integral
\begin{equation*}\label{conv}
    (\varphi \ast \psi)(\bsx) = \int_{\R^d} \varphi(\bsx-\bsy)\psi(\bsy)\rd\bsy\,.
\end{equation*}
If $\varphi,\psi \in \mathcal{S}(\R^d)$, then $\varphi \ast \psi$ still belongs to
$\mathcal{S}(\R^d)$. 
In fact, the convolution operator can be extended to $\mathcal{S}(\R^d)\times L_1$, 
in which case
we have $\varphi \ast \psi\in\S(\R^d)$, and to 
$\mathcal{S}(\R^d)\times \mathcal{S}'(\R^d)$ via
$(\varphi\ast f)(\bsx) = f(\varphi(\bsx-\cdot))$. It makes sense point-wise and is 
a $C^{\infty}$-function in $\R^d$.
As usual, the Fourier transform can be extended to $\mathcal{S}'(\R^d)$
by $(\F f)(\varphi) := f(\F \varphi)$, where
$\,f\in \mathcal{S}'(\R^d)$ and $\varphi \in \mathcal{S}(\R^d)$. 
The mapping $\F:\S'(\R^d) \to \S'(\R^d)$ is a bijection.

\subsection{Function spaces}
\label{sec:spaces}

In this section we introduce the function spaces under consideration, namely, 
the Besov and Sobolev spaces of dominating mixed smoothness. 
There are several equivalent characterizations of these spaces, 
see \cite{Vyb06}. For our purposes, the most suitable 
is the characterization by local means (see \cite[Theorem~1.23]{Vyb06} or 
\cite[Section~4]{UU16}). 

\medskip

We start with the definition of the spaces that are defined on $\R^d$.
%
Let $\Psi_0,\Psi_1\in C_0^\infty(\R)$ be such that 
\begin{enumerate}
	\item[$(i)$] $|\F\Psi_0(\xi)|>0$\qquad for $|\xi|<\eps$,
	\item[$(ii)$] $|\F\Psi_1(\xi)|>0$\qquad for $\frac\eps2<|\xi|<2\eps$ and 
	\item[$(iii)$] $D^\a\F\Psi_1(0)=0$\qquad for all $0\le\a\le L$
\end{enumerate}
for some $\eps>0$. A suitable $L$ will be chosen in 
Definition~\ref{def:besov}. 
As usual, for $j\in\N$ we define 
\[
\Psi_j(x) \,=\, 2^{j-1}\, \Psi_1(2^{j-1} x), 
\]
and the ($d$-fold) tensorization
\begin{equation}\label{eq:Psi}
\Psi_{\bsm}(\bsx) \,=\, \prod_{i=1}^d \Psi_{m_i}(x_i),
\end{equation}
where $\bsm=(m_1,\dots,m_d)\in\N_0^d$ and $\bsx=(x_1,\dots,x_d)\in\R^d$. 

\begin{remark}\label{rem:up} \rm  There exist compactly supported functions $\Psi_0, \Psi_1$ satisfying {\it (i)-(iii)} above. 
Consider $\Psi_0$ to be the {\em up-function}, 
see Rvachev~\cite{Rv90}. This function satisfies $\Psi_0\in C^{\infty}_0(\R)$ with $\supp(\Psi_0)=[-1,1]$
and $\F\Psi_0(\xi) = \prod_{k=0}^\infty\, \sinc(2^{-k}\xi)$, $\xi\in\R$, where $\sinc$ denotes the normalized {\it sinus cardinalis} $\sinc(\xi) = \sin(\pi \xi)/(\pi \xi)$. If we define $\Psi_1\in C_0^\infty(\R)$ to be
\[
\Psi_1(x) \,:=\, \frac{\rd^L}{\rd x^L}\bigl(2\Psi_0(2\ \cdot)-\Psi_0\bigr)(x), 
\qquad x\in\R\,,
\]
it follows that
$
\F\Psi_1(\xi) \,=\, (2\pi \texttt{i} \xi)^{L}
\bigl(\F\Psi_0(\xi/2) - \F\Psi_0(\xi)\bigr).
$
It is easily checked that these functions satisfy the conditions {\it (i)-(iii)} above. In particular, {\it (i)} and {\it (ii)} are
satisfied with $\eps = 1$.
Moreover, we have for all $\bsm\in\N_0^d$ that the tensorized functions 
$\Psi_{\bsm}$ satisfy $\supp(\Psi_{\bsm})\subseteq[-1, 1]^d$. 
We will work with this choice in the sequel.
\end{remark}

Let us continue with the definition of the Besov spaces $\Bspt = \Bspt(\R^d)$ 
defined on the entire $\R^d$.

\begin{definition}[Besov space]\rm \label{def:besov}
Let $0 < p,\theta\le\infty$, $s\in\R$, and 
$\{\Psi_{\bsm}\}_{\bsm\in\N_0^d}$ be as above with $L+1>s$. 
The \emph{Besov space of dominating mixed smoothness} 
$\Bspt=\Bspt(\R^d)$ is the set of all $f\in \S'(\R^d)$ 
such that 
\[
\|f\|_{\Bspt} \,:=\, 
\Big(\sum_{\bsm\in\N_0^d} 2^{s |\bsm|_1 \theta}\, 
	\|\Psi_{\bsm} \ast f\|_p^\theta\Big)^{1/\theta} \,<\,\infty
\]
with the usual modification for $\theta=\infty$. 
\end{definition}

\begin{remark}\rm 
Other important function spaces are 
\emph{Sobolev spaces of dominating mixed smoothness}, which 
are denoted by $\Hsp=\Hsp(\R^d)$ ($1<s<\infty$).
In the case $s\in\N$ these spaces can be normed by
\[
\|f\|_{\Hsp} \,:=\, \Big(\sum_{\substack{\bsell\in\N_0^d\\ |\bsell|_\infty\le s}} \|D^{\bsell} f\|_p^p\Big)^{1/p}\,.
\]
Note that $\Hs = \B^s_{2,2}$, i.e., the Sobolev(-Hilbert) spaces $\Hs$ appear 
as the special case $p=\theta=2$ in the Besov space scale $\Bspt$, 
see e.g.~\cite[Chapt.\ 2]{ScTr87}. Moreover, the spaces $\B^s_{p,p}$ with $1\le p<\infty$ and $s\notin\N$ are 
called \emph{Sobolev-Slobodeckij spaces}.
\end{remark}

\begin{remark} \rm Different choices of $\Psi_0, \Psi_1$ in Definition~\ref{def:besov} 
lead to equivalent (quasi-)norms. 
In fact, it is not even necessary that $\Psi_0$ and $\Psi_1$ have compact support. 
However, for the proof of our results this specific choice is crucial.
\end{remark}

\bigskip

In this article we are interested in classes of functions
which are supported in the unit cube $[0,1]^d$,
i.e.~we consider the subclasses of $\Bspt(\R^d)$ and $\Hsp(\R^d)$
\begin{equation}\label{def:Bo}
\Bo \,:=\, \bigl\{ f\in \Bspt(\R^d)\colon \supp(f)\subset[0,1]^d\bigr\}
\end{equation}
and 
\begin{equation}\label{def:Ho}
\Ho \,:=\, \bigl\{ f\in \Hsp(\R^d)\colon \supp(f)\subset[0,1]^d\bigr\}.
\end{equation}

\bigskip

\noindent The next lemma collects some frequently used embeddings between 
Besov and Sobolev spaces. They will be useful in obtaining our results 
for Sobolev spaces directly from the results for Besov spaces.

\begin{lemma}\label{emb} Let $p \in (0,\infty]$, and let
$s\in \R$.
\begin{enumerate}
\item[(i)] We have the chain of embeddings
$$
    \B^s_{p,\min\{p,2\}} \;\hookrightarrow\; \Hsp \;\hookrightarrow\;
\B^s_{p,\max\{p,2\}}.
$$
\item[(ii)] For $p\ge2$ we have the embedding 
$$
\Ho \;\hookrightarrow\; \mr{\H}^s_2 \;=\; \mr{\B}^s_{2,2}.
$$
\end{enumerate}
\end{lemma}

\medskip
\begin{proof} 
For a proof we refer to \cite[Chapt.\ 2]{ScTr87}. 
Note that for {\it (ii)} the compact support of the functions is necessary to 
ensure the corresponding embeddings of the $L_p$ spaces.\\
\end{proof}

In the sequel we will always assume that $s> 1/p$. This ensures that the
functions in $\Bspt$ and $\Hsp$, respectively, are continuous, see
\cite[Chapt.\ 2]{ScTr87}.

\subsection{Useful lemmas}

Here we collect some lemmas that will be essential for our analysis. 
For proofs and further literature for these results we refer 
to \cite[Sections~3.2~\&~3.3]{UU16}.

The first lemma is a variant of Poisson's summation formula, 
see~\cite[Thm.~VII.2.4]{SW71}. 
Although this equality looks more technical than the original summation 
formula, this variant is exactly in the form we need and it comes with 
less assumptions on the involved functions.

\begin{lemma}[Poisson summation formula]\label{lem:poisson}
Let $f\in L_1(\R^d)$ be continuous with compact support, 
$T\in\R^{d\times d}$ be an invertible matrix, and $B=(T^{-1})^\top$.
Furthermore, let $\phi_0\in C_0^\infty(\R)$ with $\phi_0(0)=1$ and define 
$\phi_j(t):=\phi_0(2^{-j}t)-\phi_0(2^{-j+1}t)$, $j\in\N$, $t\in\R$, as well as 
the (tensorized) functions $\phi_{\bsm}(\bsx):=\phi_{m_1}(x_1)\cdots\phi_{m_d}(x_d)$,
$\bsm\in\N_0^d$, $\bsx\in\R^d$. Then
\[
|\det(T)|\sum_{\bsell\in \Z^d} f(T\bsell) 
\,=\, \lim_{M\to\infty}\;\sum_{\bsm\colon|\bsm|_\infty\le M}\; \sum_{\bsk\in\Z^d} 
\phi_{\bsm}(B\bsk)\, \F f(B \bsk).
\]
In particular, the limit on the right hand side exists.
\end{lemma}

The lemma above holds for quite general choices of $\phi_0$. 
However, we choose a certain $\phi_0$ (and hence $\phi_m$) that is related 
to the definition of our spaces, i.e., the functions $\Psi_0$, $\Psi_1$ from 
Section~\ref{sec:spaces} are involved. 
This is because, given $\Psi_0$ and $\Psi_1$, we can construct a 
suitable \emph{decomposition of unity}.

\begin{lemma}\label{lem:decomp}
Let $\Psi_0, \Psi_1\in\S(\R)$ be functions with
\[
|\F\Psi_0(\xi)|>0\quad \text{ for } |\xi|<\eps
\]
and 
\[
|\F\Psi_1(\xi)|>0\quad \text{ for } \frac\eps2<|\xi|<2\eps
\]
for some $\eps>0$. Then there exist $\Lambda_0, \Lambda_1\in \S(\R)$ such that
\begin{enumerate}[(\it i)]
	\item $\supp\,\F\Lambda_0 \subset \{t\in \R:|t| \leq \eps\}$,
	\item $\supp\,\F\Lambda_1 \subset \{t\in \R:\eps/2 \leq |t| \leq 2\eps\}$, 
	\item $\phi_0:=\F\Lambda_0\cdot \F\Psi_0\in C_0^\infty(\R)$ with $\phi_0(0)=1$
	\item $\phi_j(\cdot):=\phi_0(2^{-j}\, \cdot)-\phi_0(2^{-j+1}\, \cdot)=\F\Lambda_j\cdot \F\Psi_j$, $j\in\N$,
				where $\Psi_j(x)=2^{j-1}\Psi_1(2^{j-1}x)$ and 
				$\Lambda_j(x)=2^{j-1}\Lambda_1(2^{j-1}x)$.
\end{enumerate}
\end{lemma}
\smallskip

\begin{proof}
Following \cite[Thm.\ 1.20]{Vyb06}, see also~\cite[Lemma~3.6]{UU16}, 
we use the special dyadic 
decomposition of unity with $\varphi(t) = 1$ if $|t|\leq 4/3$ and 
$\varphi(t) =0$ if $|t|>3/2$. 
Put $\Phi_0:=\F^{-1}\varphi$ and $\Phi_1:=2\Phi_0(2\, \cdot)-\Phi_0$, 
i.e.~$\F\Phi_1=\F\Phi_0(\cdot\, /2)-\F\Phi_0$. 
With $\Phi_j:=2^{j-1}\Phi_1(2^{j-1}\, \cdot)$ for $j\geq 1$ we define $\Lambda_0, \Lambda_1$ through 
$$
    \F\Lambda_j(t) := \frac{\F \Phi_j(2t/\eps)}{\F \Psi_j(t)}\quad \quad \mbox{for $t\in \R$.}
$$    
Then $(i)$ and $(ii)$ follow from the support of $\F\Phi_0$ and $\F\Phi_1$, 
$(iii)$ comes from $\phi_0(0)=\F\Phi_0(0)=\phi(0)=1$, 
and $(iv)$ is easily shown by using the definition of $\Phi_j$, $j\ge1$.
\end{proof}

We define the $d$-fold tensorized functions 
$$
\Lambda_{\bsm}(\bsx) \,:=\, \prod_{i=1}^d \Lambda_{m_i}(x_i)\quad\mbox{and}\quad 
\Psi_{\bsm}(\bsx) \,:=\, \prod_{i=1}^d \Psi_{m_i}(x_i)\quad \quad \mbox{for $\bsx\in \R^d$ and $\bsm\in\N_0^d$,} 
$$
where $\Lambda_j, \Psi_j$, $j\in\N_0$, are defined in Lemma~\ref{lem:decomp}.
We obtain from Lemma~\ref{lem:decomp} that we can use the functions
\begin{equation}\label{eq:decomp}
\phi_{\bsm}(\bsx) \,:=\, \F\Lambda_{\bsm}(\bsx)\cdot \F\Psi_{\bsm}(\bsx) \qquad \mbox{for $\bsx\in\R^d$,}
\end{equation}
in Lemma~\ref{lem:poisson}. 
Moreover, by the construction of the tensorized functions and Lemma~\ref{lem:decomp}, 
we know that the support of $\F\Lambda_{\bsm}$, $\bsm\in\N_0^d$, is of the form 
\begin{equation}\label{eq:dyadic}
I_{\bsm} \,:=\, \supp\,\F\Lambda_{\bsm} \,\subset\, \left\{\bsx\in\R^d\colon 
	\frac12 \lfloor2^{m_i-1}\rfloor\le|x_i|\le2^{m_i},i=1,...,d \right\}.
\end{equation}
For this, recall that we choose $\Psi_0, \Psi_1$ such that $\eps=1$, 
see Remark~\ref{rem:up}.\\

As one can see in the lemmas above we are concerned with sums of certain 
function evaluations of Fourier transforms. 
Finally, we want to bound such sums by the norm of the involved functions, 
i.e., we have to control the dependence on the matrix $B$, see Lemma~\ref{lem:poisson}. 
For this we need the following two lemmas, see~\cite[Lemmas~3.3~\&~3.5]{UU16}.

\begin{lemma}\label{lem:norm1}
Let $B\in\R^{d\times d}$ be an invertible matrix and 
$\Omega\subset\R^d$ be a bounded set. 
Furthermore, let $f\in \S(\R^d)$ with 
$\supp(f)\subset\Omega$ and define
\[
M_{B,\Omega} \,:=\, |\{\bsell\in\Z^d\colon (\bsell+[0,1)^d)\cap B^\top(\Omega)\neq\emptyset\}|.
\] 
Then, for $1\le p\le\infty$, we have
\[
\Big\|\sum_{\bsell\in\Z^d}\F f(B \bsell)\, e_{\bsell}\Big\|_{L_{p}([0,1]^d)}
\,\le\, \Big(\frac{M_{B,\Omega}}{|\det(B)|}\Big)^{1-1/p}\, 
	\|f\|_{p}.
\]
\end{lemma}

\begin{remark}\label{rem:covering}\rm
Note that $M_{B,\Omega}$ is the number of unit cubes in the standard tessellation 
of $\R^d$ that are necessary to 
cover the set $B^\top(\Omega)$, while $|\det(B)|$ equals the volume of 
$B^\top([0,1]^d)$. 
%
\end{remark}

\begin{lemma}\label{lem:norm2}
Let $B\in\R^{d\times d}$ be an invertible matrix and 
$\Omega\in\R^d$ be a bounded set. 
Furthermore, let $g\in\S(\R^d)$ with 
$\supp(\F g)\subset\Omega$.
Then, for $1\le p\le\infty$, we have
\[
\Big\|\sum_{\bsell\in\Z^d} \F g(B \bsell)\, e_{\bsell} \Big\|_{L_{p}([0,1]^d)}
\,\le\, |B(\Z^d)\cap\Omega|^{1-1/p}\, \|g\|_1.
\]
\end{lemma}

\section{General error bound}\label{sec:genbd}
In this section we provide a general upper bound on the error of general lattice 
rules for the integration in the spaces $\Bo$ of functions with support 
inside the unit cube. 
Here we mean by \emph{general lattice rules} all algorithms of the form
\begin{equation}\label{eq:latticerule}
Q_{T}(f) \,=\, |\det(T)| \sum_{\bsx\in T(\Z^d)\cap[0,1)^d} f(\bsx)
\end{equation}
with an invertible matrix $T\in\R^{d\times d}$. 
Clearly, due to $\supp(f)\subset[0,1]^d$, we could also sum over all 
$\bsx\in T(\Z^d)$ without changing the value of $Q_T(f)$.

For an invertible matrix $T$ the dual lattice of the lattice $\X=T(\Z^d)$ is defined as $\X^*=B(\Z^d):=T^{-\top}(\Z^d)$.

\begin{remark}\label{rem:latticerule} \rm
In the literature usually only those algorithms of the 
form~\eqref{eq:latticerule} are called lattice rule that satisfy 
$T(\Z^d)\supset\Z^d$, see e.g.~\cite{niesiam,sloejoe}. Such rules are a priori 
also suited for the integration of periodic functions, in contrast to 
the general algorithms of the form~\eqref{eq:latticerule}.
We decided to use this nomenclature with the prefix ``general'' since it seems to be adequate for cubature 
rules that use function evaluations on a lattice.
\end{remark}

We prove the following theorem.

\begin{theorem} \label{thm:general}
Let $T\in\R^{d\times d}$ be invertible, $\X=T(\Z^d)$, $\X^*$ its dual lattice, $I_{\bsm}$ from~\eqref{eq:dyadic} and $Q_T$ as in~\eqref{eq:latticerule}.
Then, for $1\le p,\theta\le\infty$ and $s>1/p$, 
we have
\[
{\rm wce}(Q_T, \Bo) \,\lesssim\, \left(\sum_{\bsm\in\N_0^d} 
	2^{-s |\bsm|_1 \theta'}\,|\X^*\cap I_{\bsm}\setminus\{\bszero\}|^{\theta'/p}\right)^{1/\theta'},
\]
with $\theta'=\theta/(\theta-1)$.
The hidden constant only depends on the quantity $M_{B,[-1,2]^d}/|\det(B)|$, 
cf. Remark~\ref{rem:covering}.
\end{theorem}
\medskip

\begin{proof}
From Lemma~\ref{lem:poisson} with $\phi_{\bsm}$ from~\eqref{eq:decomp} we have 
\[
Q_T(f)
\,=\, \sum_{\bsm\in\N_0^d}\; \sum_{\bsk\in\Z^d} 
\F[\Lambda_{\bsm}\ast\Psi_{\bsm}](B \bsk)\, \F f(B \bsk),
\]
where $B=T^{-\top}$. 
Actually, the outer sum is defined 
as a certain limit,  
however, we will see that this sum converges absolutely, which justifies this notation. 

Using that $\ls e_{\bsk}, e_{\bsell} \rs_{L_2([0,1]^d)}=1$, if $\bsk=\bsell$, 
and 0 otherwise, where $\ls\cdot,\cdot\rs_{L_2([0,1]^d)}$ is the usual inner product in $L_2([0,1]^d)$, 
and $I(f)=\F f(\bszero)$ we obtain
\[
\begin{split}
|I(f)-Q_T(f)| 
\,&=\, \abs{\sum_{\bsm\in\N_0^d}\,\sum_{\bsk\neq \bszero} \F\Lambda_{\bsm}(B \bsk)\, \F\Psi_{\bsm}(B \bsk)\, \F f(B \bsk)}\\
&=\,\abs{\sum_{\bsm\in\N_0^d}\,\sum_{\bsk\neq \bszero}\,\sum_{\bsell\neq \bszero} \F\Lambda_{\bsm}(B \bsk)\, 
	\F[\Psi_{\bsm}\ast f](B \bsell)\; \langle e_{\bsk}, e_{\bsell}\rangle_{L_2([0,1]^d)}} \\
&=\,\abs{\sum_{\bsm\in\N_0^d} \lsb\sum_{\bsk\neq \bszero} \F\Lambda_{\bsm}(B \bsk)\, e_{\bsk},
	\sum_{\bsell\neq \bszero}\F[\Psi_{\bsm}\ast f](B \bsell)\, e_{\bsell} \rsb_{L_2([0,1]^d)}}.\\
\end{split}
\]
By H\"older's inequality as well as Lemma~\ref{lem:norm1} and Lemma~\ref{lem:norm2} 
we have
\[
\begin{split}
|I(f)-Q_T(f)| \,&\le\, \sum_{\bsm\in\N_0^d}\, 
	\Big\|\sum_{k\neq \bszero} \F\Lambda_{\bsm}(B \bsk)\, e_{\bsk}\Big\|_{L_{p'}([0,1]^d)}\, 
	\Big\|\sum_{\bsell\neq \bszero}\F[\Psi_{\bsm}\ast f](B \bsell)\, e_{\bsell} \Big\|_{L_{p}([0,1]^d)}\\
\,&\le\, \sum_{\bsm\in\N_0^d}\, |\X^*\cap I_{\bsm}\setminus\{\bszero\}|^{1-1/{p'}} 
		\|\Lambda_{\bsm}\|_1 \,\Big(\frac{M_{B,\supp(\Psi_{\bsm}\ast f)}}{|\det(B)|}\Big)^{1-1/p}\,
		\|\Psi_{\bsm}\ast f\|_p
\end{split}
\]
with $1/p+1/p'=1$.
For this note that, by construction, $\F\Lambda_{\bsm}$ and $\Psi_{\bsm}\ast f$ have 
compact support. 
Further $\supp(\Psi_{\bsm}\ast f) \subset [-1,2]^d$ holds for all $\bsm$, see Remark~\ref{rem:up}.
Since, again by construction, $\|\Lambda_{\bsm}\|_1\lesssim1$, we finally obtain
\[
\begin{split}
|I(f)-Q_T(f)| \,&\lesssim\, \left(\sum_{\bsm\in\N_0^d}\, 2^{-s|\bsm|_1\theta'} 
	|\X^*\cap I_{\bsm}\setminus\{\bszero\}|^{\theta'/{p}} \right)^{1/\theta'}\,
	\|f\|_{\Bo}
\end{split}
\]
with $1/\theta+1/\theta'=1$, see Definition~\ref{def:besov}. \\
\end{proof}

This shows that, in order to prove bounds on the worst-case error, it is 
enough to study the numbers $|\X^*\cap I_{\bsm}\setminus\{\bszero\}|$ for $\bsm\in\N_0^d$. 
Moreover, as the next lemma shows, this quantity can be bounded by the 
Zaremba index of the lattice $\X$, which is, loosely speaking, the 
largest $\ell\in\N$ such that $|\X^*\cap I_{\bsm}\setminus\{\bszero\}|=0$ for all 
$\bsm$ with $|\bsm|_1\le\ell$.
More precisely, we define the \emph{Zaremba index} of the lattice $\X$ by
\begin{equation}\label{eq:Zaremba}
\rho(\X) \,:=\, \inf_{\bsz\in\X^*\setminus\{\bszero\}}\, r(\bsz), 
\end{equation}
where, for $\bsz=(z_1,\ldots,z_d)$, $r(\bsz):=\prod_{j=1}^d \ol{z}_j$ with $\ol{z}:=\max\{1,|z|\}$.
The connection of the Zaremba index to numerical integration is well-known, 
see e.g.~\cite{niesiam,sloejoe} and the references therein, but we repeat the (relatively short) proofs here 
for convenience of the reader.

\begin{lemma}\label{lem:points}
Let $T\in\R^{d\times d}$ be invertible, $\X=T(\Z^d)$, $\X^*$ its dual lattice and $I_{\bsm}$ from~\eqref{eq:dyadic}. Then, for all $\bsm\in\N_0^d$, we have
\[
|\X^*\cap I_{\bsm}\setminus\{\bszero\}| \,\lesssim\, \begin{cases}
0 & \text{ if }\; |\bsm|_1< \log\bigl(\rho(\X)\bigr), \\
2^{|\bsm|_1}/\rho(\X) 
& \text{ otherwise,}
\end{cases}
\]
where $\log$ denotes the dyadic logarithm.
\end{lemma}

\begin{proof}
Let $M:=\log \bigl(\rho(\X)\bigr)$.
For $\bsx\in I_{\bsm}$, 
we have $r(\bsx) \le 2^{|\bsm|_1}$. This shows that, 
for $|\bsm|_1<M$, there is no $\bsx\in\X^*$ in $I_{\bsm}$, except possibly the origin. 
This proves the first bound.

The same applies to the boxes
$
\{\bsx\in\R^d\colon |x_i|\le 2^{\ell_i}, i=1,...,d\} 
$
with $|\bsell|_1<M$.
%
If we halve all sides of this set, i.e.~if we consider the sets
\[
J_{\bsell} \,:=\, \{\bsx\in\R^d\colon 0\le x_i\le2^{\ell_i}, i=1,...,d\}, 
\]
it is easy to see that all translates of $J_{\bsell}$, $|\bsell|_1<M$, 
contain at most one $\bsx\in\X^*$.

Now consider $|\bsm|_1\ge M$. Then there exists an $\bsell\in\N_0^d$ 
with $|\bsell|_1=\lceil M\rceil-1<M$ such that $\bsm-\bsell\in\N_0^d$.
Clearly, $I_{\bsm}\subset\{\bsx\in\R^d\colon |x_i|\le2^{m_i}, i=1,...,d\}$
can be covered by $2^{|\bsm-\bsell|_1+d}=2^{|\bsm|_1-\lceil M\rceil+d+1} \le 2^{|\bsm|_1+d+1}/\rho(\X)$ 
translates of $J_{\bsell}$, each containing at most one $\bsx\in\X^*$. 
This proves the second bound.\\
\end{proof}

Together with Theorem~\ref{thm:general} this implies the following.

\begin{theorem} \label{thm:general2}
Let $T\in\R^{d\times d}$ be invertible with $d_T:=1/|\det(T)|\ge1$, $\X=T(\Z^d)$ and $Q_T$ as in~\eqref{eq:latticerule}.
Then, for $1\le p,\theta\le\infty$ and $s>1/p$, we have
\[
{\rm wce}(Q_T, \Bo) \,\lesssim\, \rho(\X)^{-s} \bigl(1+\log(d_T)\bigr)^{(d-1)(1-1/\theta)}.
\]
The hidden constant only depends on the quantity $M_{B,[-1,2]^d}/|\det(B)|$, 
see Remark~\ref{rem:covering}.
\end{theorem}
\medskip

\begin{proof}
Let $M:=\log\bigl(\rho(\X)\bigr)$ and note that 
$|\{\bsm\in\N_0^d\colon |\bsm|_1=\ell\}|\le(\ell+1)^{d-1}$.
By Theorem~\ref{thm:general} and Lemma~\ref{lem:points} we have
\[\begin{split}
{\rm wce}(Q_T, \Bo) \,&\lesssim\, \left(\sum_{\bsm\in\N_0^d} 
	2^{-s |\bsm|_1 \theta'}\,|\X^*\cap I_{\bsm}\setminus\{\bszero\}|^{\theta'/p}\right)^{1/\theta'}\\
\,&\lesssim\, \rho(\X)^{-1/p}\left(\sum_{\bsm\colon|\bsm|_1\ge M} 
	2^{(1/p-s) |\bsm|_1 \theta'}\right)^{1/\theta'}\\
\,&\le\, \rho(\X)^{-1/p}\left(\sum_{\ell=\lceil M\rceil}^\infty 
	(\ell+1)^{d-1}\,2^{(1/p-s) \ell \theta'}\right)^{1/\theta'} \\
\,&\lesssim\, \rho(\X)^{-1/p}\;2^{M (1/p-s)}\,(M+1)^{(d-1)/\theta'}\left(\sum_{\ell=1}^\infty 
	(\ell+1)^{d-1}\,2^{(1/p-s) \ell \theta'}\right)^{1/\theta'}.
\end{split}\]
It follows from Minkowski's theorem that $\rho(\X)\le|\det(T^{-1})|$, and hence $M\le\log(d_T)$. 
This proves the result since $s>1/p$.\\
\end{proof}


\section{Application to specific point sets}
\label{sec:application}

In this section we apply the general results from the last section 
to specific point sets. More precisely, we study Kronecker lattices and rank-1 lattice point sets in dimensions 
$d\ge2$. By Theorem~\ref{thm:general2} it is enough to bound the Zaremba index 
of these lattices. However, since the bounds on the Zaremba index are 
worse in higher dimensions, we treat the lower-dimensional cases separately.

\subsection{Kronecker lattices}\label{subsec:Kronecker}
We study {\it Kronecker lattices} which are
point sets of the form
\begin{equation}\label{eq:Pa}
\cP_N(\bsalpha)
\,=\, \left\{ \left(\frac{n}{N}, \{n \alpha_1\}, \ldots, \{n \alpha_{d-1}\}\right) \ : \ n = 0, 1, \ldots, N-1\right\},
\end{equation}
where $\bsalpha=(\a_1,...,\a_{d-1})\in\R^{d-1}$. These point sets can be written as 
\[
\cP_N(\bsalpha) \,=\, \X_N(\bsalpha)\cap[0,1)^d,
\]
where $\X_N(\bsalpha)=T_N(\Z^d)$ and
\begin{equation}\label{eq:Ta}
T_N=T_N(\bsalpha) = \begin{pmatrix} 
0 & \ldots & \ldots & 0 & 1/N \\
1 & 0        & \ldots & 0 & \alpha_1 \\
0 & \ddots & \ddots & \vdots & \vdots \\
\vdots & \ddots & \ddots & 0 & \vdots \\
0 & \ldots & 0 & 1 & \alpha_{d-1}
\end{pmatrix} \in \mathbb{R}^{d \times d}.
\end{equation}
Note that $\det(T_N)=(-1)^{d+1}/N$ and hence $d_{T_N}=N$. 
Hence, we can use the results from the last section to prove upper bounds 
on the error of the cubature rule $Q_{T_N}$.
Recall that the dual lattice of $\X_N(\bsalpha)$ is 
$\X_N^*(\bsalpha):=B_N(\Z^d)$ with 
\begin{equation}\label{eq:Ba}
B_N=B_N(\bsalpha)= (T_N^{\top})^{-1} = \begin{pmatrix}
-N\a_1 & -N\a_2 & \ldots & -N\a_{d-1} & N \\
1 & 0 & \ldots & 0 & 0 \\
0 & \ddots & \ddots & \vdots &  \vdots \\
\vdots & \ddots & \ddots & 0 & 0 \\
0 & \ldots & 0 & 1 & 0
\end{pmatrix}.
\end{equation}

\begin{remark}\rm 
\label{remark:bounds_on_M}
In view of Remark~\ref{rem:covering} the sequence of matrices $B_N$, 
$N\ge1$, from \eqref{eq:Ba} satisfies
\[
\sup_{N \ge 1}\, \frac{M_{B_N,[-1,2]^d}}{|\det(B_N)|} \,\le\, c_{\bsalpha} \,<\, \infty.
\]
Thus we can apply Theorem \ref{thm:general2} for the Kronecker lattices, ignoring the hidden constants.
To see this note that $|\det(B_N)|=N$ and that $B_N^\top([-1,2]^d)$ is a $d$-dimensional 
oblique prism with translated copies of $[-1,2]^{d-1} \times \{0\}$ as base 
faces. The translation vector is $(N\alpha_1,\ldots,N\alpha_{d-1},-N)$ for the ``bottom'' base and $(-2 N \alpha_1,\ldots,-2 N \alpha_{d-1},2N)$ for the ``top'' base and hence the height is $3N$ (see Figure~\ref{f1} for $d=2$). 
\begin{figure}[htp]
\begin{center}
\input{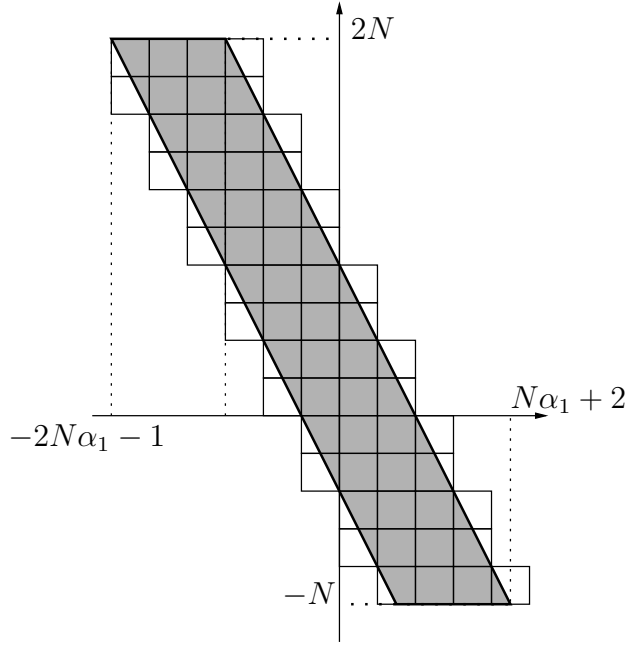}
\caption{Covering of $B_N^\top([-1,2]^2)$ with unit-squares.}
     \label{f1}
\end{center}
\end{figure}
\end{remark}

Since all lattices that follow will be of this form we use the notation
\[
\rho_d(N,\bsalpha) \,:=\, \rho\bigl(\X_N(\bsalpha)\bigr).
\]
Moreover, we write
\[
{\rm wce}\bigl(\cP_N(\bsalpha), \Bo\bigr) \,:=\, {\rm wce}\bigl(Q_{T_N}, \Bo\bigr).
\]

In the following we fix a vector $\bsalpha\in\R^{d-1}$ and just let $N$ grow 
in~\eqref{eq:Pa}. 
Given a ``good'' $\bsalpha$, this makes the point set $\cP_N(\bsalpha)$ particularly easy to implement, 
since one needs only $N d$ arithmetic operations. 

It is not surprising, and well known, that bounds on the Zaremba index of lattices $\X_N(\bsalpha)$  
depend on Diophantine approximation properties of the vector $\bsalpha$. 
More precisely, a lattice has large Zaremba index if the involved numbers 
$\a_1,...,\a_{d-1}$ are badly approximable (in a certain sense). 
This is reflected by the following lemma.

\begin{lemma}\label{lem:Zaremba}
Let $\bsalpha\in\R^{d-1}$ and $\psi: \NN \rightarrow \RR^+$ be nondecreasing 
such that
\begin{equation}\label{eq:prop-a}
\left(\prod_{j=1}^{d-1} \overline{k}_j\, \psi(\overline{k}_j) \right)  
	\langle \bsalpha \cdot \bsk\rangle \ge c=c(d,\bsalpha)>0
\end{equation}
for all $\bsk=(k_1,\ldots,k_{d-1}) \in \Z^{d-1}\setminus\{\bszero\}$, 
where $\ls x\rs:=\min_{m\in\Z}|x-m|$ is the distance of $x\in\R$ to the nearest integer. 
Then, with $c'=\min\{c, \psi(1)^{d-1}\}$, we have 
\[
\rho_d(N,\bsalpha) \,\ge\, \frac{c' N}{\psi(N)^{d-1}}\ \ \ \ \mbox{for all $N\ge1$.}
\]
\end{lemma}

\begin{proof}
To bound $\rho_d(N,\bsalpha)$, we have to bound $r(\bsz)$ uniformly over all
$\bsz\in\X^*_N(\bsalpha)\setminus\{\bszero\}$, see~\eqref{eq:Zaremba}.
By definition of $\X^*_N(\bsalpha)$ and $B_N$, we have for 
$\bsz:=B_N \bsk'\in\X^*_N(\bsalpha)\setminus\{\bszero\}$, $\bsk'=(k_1,...,k_d)\in\Z^d$, that
\[
r(\bsz) \,=\, \prod_{j=1}^d \ol{(B_N\bsk')_j}
\,=\, \left(\prod_{j=1}^{d-1}\, \overline{k}_j\right) \, 
\ol{(-N\a_1 k_1-\dots -N\a_{d-1} k_{d-1} + Nk_d)}.
\]
Since $\bsz\neq\bszero$ we have $\bsk'\neq\bszero$. We distinguish three cases. 

First assume that $k_1=...=k_{d-1}=0$ and hence $k_d\not=0$. 
This already implies $r(\bsz)\ge N$. 

Next assume that $|k_j|\ge N$ for some $j=1,...,d-1$. Then it follows from the definition that 
$r(\bsz)\ge \ol{k}_j\ge N$.

Finally, assume that 
$\bsk:=(k_1,\ldots,k_{d-1}) \in \Z^{d-1}\setminus\{\bszero\}$ with $|\bsk|_\infty\le N$. 
Clearly, by choosing the right $k_d\in\Z$, we have by assumption
\[\begin{split}
r(\bsz) \,&\ge\, N \left(\prod_{j=1}^{d-1}\, \overline{k}_j\right) \, |-\a_1 k_1-\dots -\a_{d-1} k_{d-1} + k_d|
\,\ge\, N\,\left(\prod_{j=1}^{d-1}\, \overline{k}_j\right) \ls\bsalpha\cdot\bsk\rs \\
\,&\ge\, \frac{c N}{\prod_{j=1}^{d-1}\psi\bigl(\overline{k}_j\bigr)}
\,\ge\, \frac{c N}{\psi(N)^{d-1}}.
\end{split}\]
This shows that in any case 
$r(\bsz)\ge\min\{N,c N/\psi(N)^{d-1}\}\ge N/\psi(N)^{d-1}\min\{\psi(1)^{d-1},c\}$ 
and thus
proves the claim. \\
\end{proof}

We now treat the cases $d=2$ and $d\ge3$ separately, since the known results 
on the existence of (simultaneously) badly approximable numbers differ in these cases.

\subsubsection{The case $d=2$} \label{subsec:Kronecker_dim2}

We say that a real number $\alpha$ is {\it badly approximable}, if there 
is a positive constant $c_0=c_0(\alpha)>0$ such that 
\[k \langle k \alpha \rangle \ge c_0 >0 \ \ \ \ 
\mbox{ for all integers $k \ge 1$.}
\] 
It is well known that an irrational number $\alpha$ is badly approximable 
if and only if the sequence $a_1,a_2,a_3,\ldots$ of partial quotients in 
the continued fraction expansion of $\alpha=[a_0;a_1,a_2,a_3,\ldots]$ is 
bounded, i.e. there is some $M=M(\alpha)>0$ such that $a_j \le M$ for all 
integers $j \ge 1$, see e.g.~\cite{Khinchin}.
For example for the golden ratio $\alpha = (1+\sqrt{5})/2$ we have that the 
continued fraction coefficients are all 1 and hence of course are also bounded.  

It is easily seen from the definition of badly approximable numbers that 
we can apply Lemma~\ref{lem:Zaremba} with $d=2$ and $\psi\equiv1$, 
which proves that, for the constant $c_0'=\min(1,c_0)$,
\[
\rho_2(N,\a) \,\ge\, c_0' N
\]
for all $N\in\N$.
This implies the following theorem.

\begin{theorem}\label{thm:Kronecker_dim2}
Let $\alpha$ be a badly approximable number, $N \in \mathbb{N}$ and 
$\mathcal{P}_N(\a)$ as in \eqref{eq:Pa}. 
Then, for $1\le p,\theta\le\infty$ and $s>1/p$,
\begin{equation*}
\mathrm{wce}(\mathcal{P}_N(\a),\Bo)  
\,\asymp\, \frac{(\log N)^{1-1/\theta}}{N^s}.
\end{equation*}
\end{theorem}

\begin{proof}
The upper bound follows from Theorem~\ref{thm:general2} and the lower 
bound was proven, e.g., in~\cite[Theorem~7.3]{UU16}.\\
\end{proof}


\subsubsection{The case $d\ge3$}

Unfortunately, in dimensions greater two the results are not as satisfactory as 
for $d=2$. That is, we do not know if vectors $\bsalpha$ exist that give the 
optimal order of convergence of the corresponding (general) lattice rule.

Assume for the moment that we have a vector $\bsalpha \in \mathbb{R}^{d-1}$ for $d \ge 3$ 
such that 
\[
r(\bsk) \langle \bsalpha \cdot \bsk \rangle >c=c(\bsalpha)>0\qquad\mbox{ for all $\bsk \in \mathbb{Z}^{d-1}\setminus\{\bszero\}$.}
\] 
In this case we could show that $\mathrm{wce}(\mathcal{P}_N(\bsalpha),\Bo) \lesssim (\log N)^{(d-1)(1-1/\theta)}/N^s$, 
which is the optimal order of convergence. 
However, a famous conjecture of Littlewood states that there is no $\bsalpha \in \mathbb{R}^{d-1}$, $d \ge 3$, 
with this property, see e.g.~\cite{badziahin}. See also \cite{beck} for a discussion of this Diophantine problem in the context of the discrepancy of $(n \bsalpha)$-sequences.

The best we can hope for at the moment for our problem are metrical results. These are based on the following lemma.

\begin{lemma}\label{lem:Zaremba2}
Let $\psi: \NN \rightarrow \RR^+$ be non-decreasing 
such that the series $\sum_{n \ge 1} \frac{1}{n \psi(n)} < \infty$. 
Then for almost every $\bsalpha \in \RR^{d-1}$ and every $N\ge1$ we have
\[
\rho_d(N,\bsalpha) \,\ge\, \frac{c N}{\psi(N)^{d-1}}
\]
for some $c>0$.
For example, we can choose $\psi(N)=(\log N) (\log\log N)^{1+\delta}$ for arbitrary $\delta>0$ for $N \ge 3$ and $\psi(N)=1$ for $N <3$.
\end{lemma}

\begin{proof}
From \cite[Lemma~5]{bhak} we obtain that, under the assumptions of the lemma, we can 
apply Lemma~\ref{lem:Zaremba} for almost every $\bsalpha\in \RR^{d-1}$. \\
\end{proof}

This implies the following result.

\begin{theorem}\label{thm:Kronecker}
Let $\psi: \NN \rightarrow \RR^+$ be non-decreasing such that 
$\sum_{n \ge 1} \frac{1}{n \psi(n)} < \infty$. Then, for almost all $\bsalpha \in \RR^{d-1}$ and every $N \ge 1$
we have 
\[
\mathrm{wce}(\mathcal{P}_N(\bsalpha),\Bo) \,\lesssim\, \frac{(\log N)^{(d-1)(1-1/\theta)} }{N^s}\, \psi(N)^{s(d-1)}.
\] 
For example, for $\delta>0$ for almost all $\bsalpha \in \RR^{d-1}$ we have 
\[
\mathrm{wce}\bigl(\mathcal{P}_N(\bsalpha),\Bo\bigr) \,\lesssim\,\frac{(\log N)^{(d-1)(s+1-1/\theta)}}{N^s} (\log \log N)^{s(d-1)(1+\delta)}.
\]
\end{theorem}

\begin{remark}\rm
In dimension $d=3$ the metrical result can be slightly improved. 
It follows from results in \cite{badziahin} that 
there exist $(\a_1,\a_2)\in\R^2$ such that the assumption of Lemma~\ref{lem:Zaremba} 
holds with $\psi(N)=(\log N) \log\log N$. Hence, for $d=3$, the second statement of 
Theorem~\ref{thm:Kronecker} holds with $\delta=0$.
\end{remark}

However, if we want a result for concrete $\bsalpha \in \RR^{d-1}$ for $d \ge 3$, the 
situation is even worse.
Recall that for a real number $\eta$, a $(d-1)$-tuple $\bsalpha \in (\RR \setminus \QQ)^{d-1}$ 
is said to be of {\it approximation type} $\eta$, if $\eta$ is the infimum of all 
numbers $\sigma$ for which there exists a positive constant $c=c(\sigma,\bsalpha)$ such 
that 
\begin{equation}\label{approxtype}
r(\bsh)^{\sigma} \langle \bsh \cdot \bsalpha\rangle \ge c\ \ \ 
\mbox{ for all }\ \bsh \in \ZZ^{d-1} \setminus \{\bszero\}.
\end{equation}
It is well known that the type $\eta$ of an irrational vector $\bsalpha$ is at least one. 
On the other hand it has been shown by Schmidt \cite{Schmid} that 
$\bsalpha=(\alpha_1,\ldots,\alpha_{d-1})$, with real algebraic components for which 
$1,\alpha_1,\ldots,\alpha_{d-1}$ are linearly independent over $\QQ$, is of type $\eta=1$. 
In particular, $({\rm e}^{r_1},\ldots,{\rm e}^{r_{d-1}})$ with distinct nonzero rationals 
$r_1,\ldots,r_{d-1}$ or $(\sqrt{p_1},\ldots,\sqrt{p_{d-1}})$ with distinct prime numbers 
$p_1,\ldots,p_{d-1}$ are of type $\eta=1$.

From \eqref{approxtype}, Lemma~\ref{lem:Zaremba} and Theorem~\ref{thm:general2} we obtain the following result.

\begin{theorem}\label{thm_dimd}
Let $\bsalpha \in (\RR \setminus \QQ)^{d-1}$ be of approximation type 1. Then for every $\delta>0$ we have
\begin{equation*}
\mathrm{wce}(\mathcal{P}_N(\bsalpha),\Bo) \,\lesssim\, \frac{1}{N^{s-\delta}}.
\end{equation*}
\end{theorem}

\subsection{Rank-1 lattice point sets}\label{subsec:rank-1}

A rank-1 lattice point set is given by the points $\{\frac{n}{N} \bsg\}$ 
for $n=0,1,\ldots,N-1$, where $\bsg=(g_0,g_1,\ldots,g_{d-1})$ is a 
lattice point in $\Z^d$ and where the fractional part is applied component wise. 
We restrict ourselves to the case where $g_0=1$ 
(if $N$ is a prime number, this still covers all possible rank-1 lattice point sets). 
Then we can write a rank-1 lattice point set as 
\begin{align*}
\cP_N(\bsg) & = \X_N(\bsg) \cap [0,1)^d\\
 & = \left\{\left(\frac{n}{N},\left\{n\, \frac{g_1}{N}\right\},\ldots,\left\{n\, \frac{g_{d-1}}{N}\right\}\right) \ : \ n=0,1,\ldots,N-1\right\},
\end{align*}
where $\X_N(\bsg)=T_N(\Z^d)$ with 
\begin{equation}\
T_N=T_N(\bsg)= \begin{pmatrix} 
0 & \ldots & \ldots & 0 & 1/N \\
1 & 0        & \ldots & 0 & g_1/N \\
0 & \ddots & \ddots & \vdots & \vdots \\
\vdots & \ddots & \ddots & 0 & \vdots \\
0 & \ldots & 0 & 1 & g_{d-1}/N
\end{pmatrix} \in \mathbb{R}^{d \times d}.
\end{equation}

In view of Section~\ref{subsec:Kronecker} we see that we replace the 
possibly irrational point $\bsalpha$ by the rational point $\bsg/N$. 
So, for consistent notation, we should have used, e.g., the denotation 
$\cP_N(\bsg/N)$ for the point set. However we use $\cP_N(\bsg)$ etc.~for simplicity.
For the same reasoning we let $\rho_d(N,\bsg):=\rho(\X_N(\bsg/N))$.
The Zaremba index of rank-1 lattice point sets is a well studied quantity and we can use the known results in order to apply them in Theorem~\ref{thm:general2}. 
Since the statement of Remark~\ref{remark:bounds_on_M} holds also in this case, 
we can ignore the hidden constants in Theorem~\ref{thm:general2}. Again we treat the cases $d=2$ and $d \ge 3$ separately.

\begin{remark}\rm \label{rem:construction}
For $d=2$ the construction is again based on the boundedness of the 
partial quotients of $g_1/N$. So, given a badly approximable number $\alpha$, 
see Section~\ref{subsec:Kronecker_dim2}, one can use its 
\emph{convergents} $p_k/q_k$, $k=1,2,...$, to construct the (optimal) 
sequence of lattices $\cP_{q_k}\bigl((1,p_k)\bigr)$. 
To find an analogous construction in higher dimensions is a challenging open problem.
\end{remark}

\subsubsection{The case $d=2$}

Let $g\in \{1,\ldots,N-1\}$ with $\gcd(g,N)=1$. Let $a_1,a_2,\ldots,a_l$ be the partial quotients in the continued fraction expansion of $g/N$ and let $K(\tfrac{g}{N})=\max_{1 \le j \le l} a_j$. Then it was shown by Zaremba~\cite{zar1} (see also \cite[Theorem~5.17]{niesiam}) that the Zaremba index $\rho_2(N,\bsg)$ for $\bsg=(1,g)$ can be bounded in terms of $K(g/N)$, more precisely, that $$\frac{N}{K(g/N)+2} \le \rho_2(N,\bsg) \le \frac{N}{K(g/N)}.$$

From this result in conjunction with Theorem~\ref{thm:general2} we obtain the following result:

\begin{theorem}\label{thm_latrule2}
Let $N \in \mathbb{N}$ and $\bsg=(1,g)$ with $g \in \{1, 2, \ldots, N-1\}$ such that $\gcd(g, N) = 1$ and $K(g/N) \le C$ for some constant $C > 0$. Then, for $1\le p,\theta\le\infty$ and $s>1/p$,
\begin{equation*}
\mathrm{wce}(\mathcal{P}_N(\bsg),\Bo)  
\,\asymp\, \frac{(\log N)^{1-1/\theta}}{N^s}.
\end{equation*}
\end{theorem}

\begin{remark}\rm 
In particular the result holds for {\it Fibonacci rules}, where $N=F_n$ and $g=F_{n-1}$, the $n^{\rm th}$ and $(n-1)^{\rm st}$ Fibonacci numbers, respectively. In this case the continued fraction coefficients of $g/N=F_{n-1}/F_n$ are all exactly 1. 
Fibonacci rules were also used by Temlyakov, see \cite{teml1993} and \cite[Section~4.1]{teml2003}.
\end{remark}

\begin{remark}\rm
Note that a famous conjecture of Zaremba \cite[p.~76]{zar} states that for 
every integer $N \ge 2$ one can find a $g\in \{1,\ldots,N\}$ with 
$\gcd(g,N)=1$ such that the continued fraction coefficients of $g/N$ are 
bounded by some constant $K$ (in fact, he conjectured that $K=5$). 
Niederreiter \cite{niezar} established this conjecture for all $N$ of the 
form $2^m, 3^m$ or $5^m$ for $m \in \mathbb{N}$. 
Bourgain and Kontorovich~\cite{BourKont} proved Zaremba's conjecture for 
almost all $N$ with a constant $K=50$. Huang~\cite{Huang} improved this 
result to show Zaremba's conjecture for almost all $N$ with constant $K = 5$.
\end{remark}

\subsubsection{The case $d \ge 3$}

It follows from a result of Zaremba~\cite{zar3} that for every $d\ge 2$ and $N \ge 2$ there exists a lattice point $\bsg \in \Z^{d}$ such that 
\begin{equation}\label{ex_zar_ind}
\rho_d(N,\bsg) \ge \frac{C_d N}{(\log N)^{d-1}}.
\end{equation}
In fact, one can choose $C_d=(d-1)!/2^{d-1}$ (see also \cite[Theorem~5.12]{niesiam}). From this result together with Theorem~\ref{thm:general2} we obtain the following theorem.

\begin{theorem}
For every $d\ge 2$ and $N \ge 2$ there exists a lattice point $\bsg \in \Z^{d-1}$ such that $$\mathrm{wce}(\mathcal{P}_N(\bsg),\Bo) \lesssim \frac{(\log N)^{(d-1)(s+1-1/\theta)}}{N^s}.$$   
\end{theorem}


However, it remains an open question how to construct, for given $d$ and $N$, 
lattice points 
$\bsg \in \Z^{d-1}$ which achieve the lower bound \eqref{ex_zar_ind}. 
Without loss of generality one can restrict to the search space 
$\{g \in \{1,2,\ldots,N-1\} \ : \ \gcd(g,N)=1\}^{d-1}$ of size 
$\phi(N)^{d-1}$, where $\phi$ denotes Euler's totient function. 
This is too large for a full search already for moderately large $N$ 
and $d \ge 3$. So far one relies on computer search to find good 
generating vectors, usually based on the fast component-by-component 
construction \cite{NC06a, NC06b}. 

Korobov~\cite{kor} suggested to consider lattice point sets with generating 
vectors $\bsg=(1,g,g^2,\ldots,g^{d-1})$ in $\Z^d$ with $g \in \{1,2,\ldots,N-1\}$ 
such that $\gcd(g,N)=1$. The search space for lattice points of this form 
reduces to $\phi(N)$. 
At least for prime powers $N$ and in dimension $d=3$ there is an existence 
result of Larcher and Niederreiter~\cite{LarNie} for $\bsg=(1,g,g^2)$ with 
$$\rho_3(N,\bsg) \ge \frac{C N}{(\log N)^2}.$$ 
For $d>3$ this is open.

\noindent {\sc Josef Dick, Kosuke Suzuki and Takehito Yoshiki} 

\noindent School of Mathematics and Statistics, The University of New South Wales, Sydney NSW 2052, Australia, 
email: josef.dick(AT)unsw.edu.au, kosuke.suzuki1(AT)unsw.edu.au, takehito.yoshiki1(AT)unsw.edu.au\\

\noindent {\sc Friedrich Pillichshammer}

\noindent Institut f\"ur Finanzmathematik und angewandte Zahlentheorie, Johannes Kepler Universit\"at Linz, Altenbergerstra{\ss}e 69, 4040 Linz, Austria,
email: friedrich.pillichshammer(AT)jku.at\\

\noindent {\sc Mario Ullrich}

\noindent Institut f\"ur Analysis, Johannes Kepler Universit\"at Linz, Altenbergerstra{\ss}e 69, 4040 Linz, Austria, 
email: mario.ullrich(AT)jku.at\\

\end{document}